\documentclass[11pt, reqno]{amsart}

\usepackage{amssymb}
\usepackage{amsmath}
\usepackage{mathrsfs}
\usepackage{amsfonts}
\usepackage{color}
\usepackage{vmargin}
\usepackage{amsthm}
\usepackage{graphicx}

\long\def\symbolfootnote[#1]#2{\begingroup%
\def\thefootnote{\fnsymbol{footnote}}\footnote[#1]{#2}\endgroup}

\setmarginsrb{20mm}{20mm}{20mm}{20mm}{10mm}{10mm}{10mm}{10mm}

\newtheoremstyle{remark}
  {}{}{}{}{\bfseries}{.}{.5em}{{\thmname{#1 }}{\thmnumber{#2}}{\thmnote{ (#3)}}}

\RequirePackage{amsthm}

\newtheorem{defi}{Definition}[section]

\newtheorem{tw}[defi]{Theorem}

\newtheorem{lem}[defi]{Lemma}
\newtheorem{cor}[defi]{Corollary}
\newtheorem{prop}[defi]{Proposition}

\newtheorem{ex}[defi]{Example}

\def\vint{\mathop{\mathchoice%
          {\setbox0\hbox{$\displaystyle\intop$}\kern 0.22\wd0%
           \vcenter{\hrule width 0.6\wd0}\kern -0.82\wd0}%
          {\setbox0\hbox{$\textstyle\intop$}\kern 0.2\wd0%
           \vcenter{\hrule width 0.6\wd0}\kern -0.8\wd0}%
          {\setbox0\hbox{$\scriptstyle\intop$}\kern 0.2\wd0%
           \vcenter{\hrule width 0.6\wd0}\kern -0.8\wd0}%
          {\setbox0\hbox{$\scriptscriptstyle\intop$}\kern 0.2\wd0%
           \vcenter{\hrule width 0.6\wd0}\kern -0.8\wd0}}%
          \mathopen{}\int}
\def\={\hspace{-3mm}&=&\hspace{-3mm}}

\newcommand{\pc}{{p(\cdot)}}
\newcommand{\rn}{\mathbb{R}^n}

\newcommand{\wfr}{W^{1,\phi}(\rn)}
\newcommand{\lfr}{L^{\phi}(\rn)}
\newcommand{\lf}{L^{\phi}}
\newcommand{\n}{\mathbb{N}}
\newcommand{\R}{\mathcal{R}}
\renewcommand{\r}{\mathbb{R}}
\newcommand{\loc}{\textnormal{loc}}
\newcommand{\fR}{\mathfrak{R}}
\newcommand{\wR}{\widetilde{R}}
\newcommand{\fiw}{\Phi_{\textnormal{w}}}
\newcommand{\fiwr}{\Phi_{\textnormal{w}}(\rn)}
\newcommand{\fic}{\Phi_{\textnormal{c}}}
\newcommand{\fis}{\Phi_{\textnormal{s}}}
\newcommand{\Az}{\textnormal{{(A0)}}}
\newcommand{\Aj}{\textnormal{{(A1)}}}
\newcommand{\Ad}{\textnormal{{(A2)}}}
\newcommand{\ainc}{\textnormal{{(aInc)}}}

\newcommand{\adec}{\textnormal{{(aDec)}}}

\let \phi \varphi

\begin{document}

\date{}

\title{\bf Maximal operator in Musielak--Orlicz--Sobolev spaces}

\author{Piotr Micha{\l} Bies, Micha{\l} Gaczkowski, Przemys{\l}aw G\'{o}rka \bigskip}

\begin{abstract}
We study the Hardy-Littlewood maximal operator in the Musielak-Orlicz-Sobolev space $W^{1,\phi}(\mathbb{R}^n)$. Under some natural assumptions on $\phi$ we show that the maximal function is bounded and continuous in $W^{1,\phi}(\mathbb{R}^n)$.
\end{abstract}
\maketitle
\bigskip

\medskip


\medskip
\section{Introduction}

The Hardy-Littlewood maximal operator $M$ plays a central role in mathematics, e.g. in the theory of function spaces, harmonic analysis and PDE. The boundedness of $M$ in various types of function spaces is a central issue. Classical theorem of Hardy, Littlewood and Wiener asserts that for $1<p \leq \infty$, the maximal operator is bounded on $L^{p}(X, d, \mu)$, where $(X, d, \mu)$ is a doubling metric measure space (see e.g. \cite{Heinonen}). The maximal operator has been also studied in different function spaces, e.g.: Banach function spaces \cite{Sh}, Sobolev spaces \cite{K}, Lebesgue spaces with variable exponent \cite{Diening}, Musielak-Orlicz spaces \cite{Hasto}. Moreover, $M$ is bounded in H\"older spaces $C^{0,s}(X)$ (see \cite{Bu}), where $(X, d, \mu)$ satisfies the so-called $\delta$-annular decay property\footnote{If no annular decay property is assumed, then $Mf$ can fail to be continuous, even if $f$ is Lipschitz continuous (see Example 1.4. in \cite{Bu}). } and $\mu$ is doubling (see also \cite{Gorka} in the space of continuous functions $C(X)$ and \cite{BGG} in the variable exponent H\"older spaces).  

Let us observe that since $M$ is sublinear and bounded in $L^p$, we get the continuity of the maximal operator in $L^p$. The same holds in the space of bounded and continuous functions $C(X)$. Iwaniec asked the following question {\cite[Question 3]{HO}}: \textit{Is the Hardy–Littlewood maximal operator continuous in $W^{1,p}(\mathbb{R}^n)$, $1 < p < \infty$?} The positive answer to this question is given in \cite{Luiro}. The main objective of the paper is to study the maximal operator in the Musielak-Orlicz-Sobolev spaces $W^{1,\phi}(\mathbb{R}^n)$. Under some natural assumptions on $\phi$ we show that maximal operator is bounded and continuous in $W^{1,\phi}(\mathbb{R}^n)$.

The remainder of the paper is structured as follows. In Section 2, we introduce the notations and recall the definitions. In Section 3 we prove boundedness of the maximal operator in the Musielak-Orlicz-Sobolev spaces. Auxiliary results, which are needed in the proof of the main result, are contained in Section 4. Our principal assertion, concerning the continuity of the maximal operator in the Musielak-Orlicz-Sobolev is formulated and proven in the last section.

\section{Preliminaries}

Let $A\subset\rn$ be a measurable set, by $|A|$ we denote the Lebesgue measure of $A$ and by $L^0(A)$ the set of measurable functions on $A$.

Let $\phi:A\times [0,\infty)\to[0,\infty)$, $p, q >0$. We say that $\phi$ satisfies $\ainc_p$ if there exists $a\in [1,\infty)$ such that the inequality holds 
$$
\frac{\phi(x,s)}{s^p}\leq a\frac{\phi(x,t)}{t^p} \quad \text{for almost all $x\in A$ and for all $0<s<t$},
$$
 and we say that $\phi$ satisfies $\adec_q$ if there exists $a\in [1,\infty)$ such that the inequality holds
$$
\frac{\phi(x,t)}{t^q}\leq a\frac{\phi(x,s)}{s^q}\quad \text{for almost all $x\in A$ and for all $0<s<t$}.
$$
 Furthermore, we denote 
\[
	\ainc\ = \bigcup_{p \in (1, \infty)} \ainc_p, \,\,\adec\ = \bigcup_{p \in (1,\infty)} \adec_p.
\]
We say that $\phi: A\times[0,\infty)\to[0,\infty]$ is a $\Phi$-prefunction if $\phi(x,0)=0$, $\phi(x,\cdot)$ is increasing, $\lim_{t\to 0^+}\phi(x,t)=0$, $\lim_{t\to\infty}\phi(x,t)=\infty$ for almost every $x\in A$ and the map $x\mapsto \phi(x,|f(x)|)$ is measurable for $f\in L^0(A)$. We say that a $\Phi$-prefunction $\phi$ is a
\begin{itemize}
\item weak $\Phi$-function if it satisfies $\ainc_1$,
\item convex $\Phi$-function if $\phi(x,\cdot)$ is left continuous and convex for almost every $x\in A$,
\item strong $\Phi$ function if $\phi(x,\cdot)$ is continuous and convex for almost every $x\in A$.
\end{itemize}
The set of weak $\Phi$-functions, convex $\Phi$-functions and strong $\Phi$-functions we shall denote by $\fiw(A)$, $\fic(A)$ and $\fis(A)$ respectively. From the very definition we 
have $\fis(A)\subset\fic(A)\subset\fiw(A)$. For $\phi\in\fiw(A)$ and $f\in L^0(A)$ we define
$$
\rho_{\phi}(f)=\int_A\phi(x,|f(x)|)\;dx,
$$ 
$$
\|f\|_{\phi,A}=\inf \left\{\lambda>0\;:\;\rho_{\phi}\left(\tfrac{f}{\lambda}\right)\leq 1\right\},
$$
and a set
$$
\lf(A)=\{f\in L^0(A)\;:\;\exists(\lambda>0)\ \rho_{\phi}(\lambda f)<\infty\}.
$$
We shall simply write $\|f\|_\phi$ when $A=\rn$. Let us note that the Musielak-Orlicz space $(\lf,\|\cdot\|_{\phi,A})$ is a quasi-Banach space for $\phi\in\fiw(A)$, and $(\lf,\|\cdot\|_{\phi,A})$ is a Banach space if $\phi\in\fic(A)$. It is known (see {\cite[Proposition 3.2.4]{genOrlicz}), that if $\phi,\psi\in\fiw(A)$ and $\phi\simeq\psi$, then\footnote{Let $\phi,\psi :A\times [0,\infty)\to[0,\infty]$. We say that $\phi$ and 
$\psi$ are equivalent ($\phi\simeq \psi$) if there exists $L\geq 1$ such that the inequalities $\psi(x,t/L)\leq \phi(x,t)\leq\psi(x,L t)$
are satisfied for almost every $x\in A$ and for all $t\in [0,\infty)$.}  $\lf(A)=L^\psi(A)$ and corresponding quasi-norms are equivalent. Moreover, if $\phi\in\fiw(A)$, then there exists $\psi\in\fis(A)$ such that $\phi\simeq\psi$. Thus, even if $\|\cdot\|_{\phi,A}$ is not a norm for a certain $\phi\in\fiw(A)$ it has a Banach space structure. 
Let us recall {\cite[Lemma 3.1.3]{genOrlicz}} that if $\phi\in\fiw(A)$ satisfies \adec, then 
$$
\lf(A)=\left\{f\in L^0(A)\,:\, \rho_{\phi}(f)<\infty\right\}.
$$
On the other hand {\cite[Theorem 3.6.6]{genOrlicz}}, if $\phi\in\fiw(A)$ satisfies \ainc\ and \adec, then $\lf(A)$ is reflexive.

If $\phi\in\fiw(A)$ satisfies $\ainc_p$ with $1\leq p<\infty$, then quantities $\|. \|_{\phi}$ and $\rho_{\phi}(. )$ can be compared using the inequalities (see {\cite[Corollary 3.2.10]{genOrlicz}})
\begin{eqnarray} \label{osznorm}
\min\{(\tfrac{1}{a}\rho_{\phi}(f))^{1/p},1\}\leq\|f\|_{\phi}\leq\max\{(a\rho_{\phi}(f))^{1/p},1\},
\end{eqnarray}
where $a$ is a constant from $\ainc_p$. Let us recall that if $\phi\in\fiw(A)$, then for $f,g\in L^0(A)$ the following H\"older inequality 
$$
\int_A |f(x)||g(x)|\; dx\leq 2\|f\|_{\phi,A}\|g\|_{\phi^*,A},
$$
holds.\footnote{For $\phi:A\times [0,\infty)\to [0,\infty]$ we define the conjugate $\phi^*$ as follows $\phi^*(x,t)=\sup_{s\geq 0}(ts-\phi(x,s))$.}

\begin{defi}
Let $A\subset\rn$ be a measurable set. We say that $\phi\in\fiw(A)$ satisfies \Az\ if there exists a constant $\beta\in (0,1]$ such that $\beta\leq\phi^{-1}(x,1)\leq{1}/{\beta}$ for almost everyl $x\in A$, where $\phi^{-1}$ is left-inverse of $\phi$ (see {\cite[Definition 2.3.1]{genOrlicz}}).
\end{defi}

\begin{prop}[{\cite[Corollary 3.7.4]{genOrlicz}}]\label{azrow}
Let $A\subset\rn$ be a measurable set and let $\phi\in\fiw(A)$. Then $\phi$ satisfies \Az\ if and only if there exists $\beta\in (0,1]$ such that 
$\phi(x,\beta)\leq 1\leq\phi(x,1/\beta) $ for almost every $x\in A$.
\end{prop}

\begin{prop}[{\cite[Corollary 3.7.9]{genOrlicz}}]\label{lfzan}
Let $A\subset\rn$ be a measurable set with finite measure. If $\phi \in\fiw(A)$ satisfies \Az and $\ainc_p$ with $1\leq p<\infty$, then $\lf(A)\hookrightarrow L^p(A)$.
\end{prop}

\begin{defi}
Let $\Omega\subset\rn$ be an open set and $\phi\in\fiw(\Omega)$, then
\begin{enumerate}
\item 
$\phi$ satisfies \Aj\ if there exists $\beta\in (0,1)$ such that for every ball $B$ such that $|B|\leq 1$ 
the following inequality holds
$$
\beta\phi^{-1}(x,t)\leq\phi^{-1}(y,t)
$$
for every $t\in [1,1/|B|]$ and for almost every $x,y\in B\cap \Omega$.
\item $\phi$ satisfies \Ad\ if for all $s>0$ there exist $\beta\in (0,1]$ and $h\in L^1(\Omega)\cap L^{\infty}(\Omega)$ such that the following inequality holds
$$
\beta \phi^{-1}(x,t)\leq\phi^{-1}(y,t)
$$
for almost every $x,y\in\Omega$ and for all $t\in [h(x)+h(y),s]$.
\end{enumerate}
\end{defi}
Next, we define  the Musielak--Orlicz--Sobolev spaces. Let $\Omega$ be an open subset of $\mathbb{R}^n$ and let $\phi\in\fiw(\Omega)$. 
The Musielak--Orlicz--Sobolev space $W^{1, \phi}(\Omega)$ is a vector space of all $f \in \lf(\Omega)$ for which the distributional derivatives
 belong to $\lf(\Omega)$. We equip $W^{1, \phi}(\Omega)$ with the quasi-norm 
\begin{eqnarray*}
\| u \|_{k,\phi,\Omega} := \sum_{|\alpha|\leq 1} \| D_\alpha u \|_{\phi,\Omega}.
\end{eqnarray*}
Again, we will write simply $\|u\|_{1,\phi}$ if $\Omega=\rn$. If $\phi\in\fiw(\Omega)$ satisfies \Az, \ainc\ and \adec, then the Musielak--Orlicz--Sobolev space is reflexive \cite[Theorem 6.1.4 and Theorem 3.7.13]{genOrlicz}.

If $f$ is locally integrable and $A$ is a measurable set such that $0<|A|<\infty$, then we denote the integral average of the function $f$ over $A$ as  
\[
  \vint_A f\,dx= \frac{1}{|A|}\int_A f\,  dx\,.
\]
Finally, we formulate and prove the following proposition.
\begin{prop}\label{propjedn}
Let $\phi \in \fiwr$ satisfies \adec\ and let $f\in\lfr$. Then
\begin{itemize}
\item[(i)] For $\epsilon>0$ there exists $R>0$ such that 
$$
\|f\|_{\phi,\rn\setminus B(0,R)}<\epsilon.
$$
\item[(ii)] For $\epsilon>0$ there exists $\lambda>0$ such that for any measurable set $A\subset\rn$ such that $|A|<\lambda$ the following inequality 
$$
\|f\|_{\phi,A}<\epsilon
$$
is satisfied.
\end{itemize}
\end{prop}
\begin{proof}
Since the proof of (ii) is similar to the proof of (i), we shall give the proof of (i). Let $\epsilon>0$, then since \adec\ is satisfied, we have
$$
\int_{\rn}\phi\left(x,\tfrac{|f(x)|}{\epsilon/2}\right)\,dx<\infty.
$$
Thus, there exists $R>0$ such that 
$$
\int_{\rn\setminus B(0,R)}\phi\left(x,\tfrac{|f(x)|}{\epsilon/2}\right)\,dx\leq 1.
$$
Therefore, from the very definition of the norm we have $\|f\|_{\phi,\rn\setminus B(0,R)}<\epsilon$.
\end{proof}
We close this section with couple of examples taken from \cite{genOrlicz}.
\begin{ex}
Let $\phi_0(x,t)=\phi(t)$, where $\phi : [0, \infty) \rightarrow [0,\infty)$ is such that  $\phi_0 \in \fiw$. Then, $\phi_0$ satisifes \Az, \Aj, \Ad. 
Moreover,  $\phi_0$ satisifes \ainc, if $\phi$ satisfies $\nabla_2$ and $\phi_0$ satisifes \adec, if $\phi$ satisfies $\Delta_2$.
\end{ex}

\begin{ex}
Let $\phi_1(x,t)=t^{p(x)}$, where $p:\Omega \rightarrow [1,\infty)$ is a measurable map, then 
$\phi_1$ satisifes \Az,  $\phi_1$ satisifes \Aj, if $\frac{1}{p}\in C^{\log}$, i.e., there exists $C$ such that for every distinct $x, y \in \Omega$ 
\[
	\left|\frac{1}{p(x)}-\frac{1}{p(y)}\right|\leq \frac{C}{\log (e +1/|x-y|)},
\]
$\phi_1$ satisifes \Ad, if $\frac{1}{p}$ satisfies $\log$-H\"older decay condition, i.e., there exist $C, p_{\infty}$ such that
\[
	\left|\frac{1}{p(x)}-\frac{1}{p_{\infty}}\right|\leq \frac{C}{\log (e +|x|)}.
\]
Furthermore, $\phi_1$ satisifes \ainc, if $p_{-}>1$ and  $\phi_1$ satisifes \adec, if $p_{+}<\infty$.
\end{ex}

\begin{ex}
Let $\phi_2(x,t)=t^{p}+a(x)t^q$, where $q>p\geq 1$, then 
 $\phi_2$ satisifes \Az, if $a \in L^{\infty}$,  $\phi_2$ satisifes \Aj, if $a \in C^{0, \frac{n}{p}(q-p)}$,  $\phi_2$ always satisifes \Ad, and  $\phi_2$ satisifes \ainc, if $p>1$ and  $\phi_2$ satisifes \adec, if $q<\infty$.
\end{ex}

\section{Boundedness of the maximal operator}
For $f\in L^1_{\loc}(\rn)$ we define the maximal function $Mf\colon\rn\to\r$ in a standard way 
$$
Mf(x)=\sup_{r>0}\vint_{B(x,r)}|f(z)|\,dz.
$$
Let us recall the crucial theorem about boudedness of maximal function in the Musielak-Orlicz spaces.

\begin{tw}[{\cite[Theorem 4.3.4]{genOrlicz}}]\label{lfmax}
Let $\phi\in\fiwr$ satisfies \Az, \Aj, \Ad\ and \ainc. Then, the maximal operator
$$
M\colon \lfr\to\lfr
$$
is bounded.
\end{tw}
Now, we are in position to formulate and prove the main result of this section.

\begin{tw}\label{lpzmaxb}
Let $\phi\in\fiwr$ satisfies \Az, \Aj, \Ad, \ainc\ and \adec. If $f\in \wfr$, then $Mf\in\wfr$ and the inequality
\begin{align}\label{lpzmaxbtez}
|D_iMf(x)|\leq MD_if(x),
\end{align}
is satisfied for all $i=1,\ldots,n$ and for almost all $x\in\rn$.\footnote{Having in mind Theorem \ref{lfmax} we get that $M$ is bunded from $\wfr$ to $\wfr$.}
\end{tw}
Let us remark that in the setting of $W^{1,p}$ spaces the above theorem  has been proven by Kinnunen \cite{K}.
\begin{proof}
For $r>0$ we denote by $\chi_{B(0,r)}$ a characteristic function of the ball $B(0,r)$. Let $$h_r=\frac{1}{|B(0,r)|}\chi_{B(0,r)}.$$ Then, we have 
$$
|f|*h_r(x)=\vint_{B(x,r)}|f(y)|\,dy\leq Mf(x).
$$
This estimate and properties of convolution yields
\[
|D_i(|f|*h_r)(x)|=|(D_i|f|)*h_r(x)|\leq M D_i|f|(x)
\]
for almost all $x\in\rn$. By Theorem \ref{lfmax} $Mf, M D_i|f|\in \lfr$, and therefore $|f|*h_r\in\wfr$. Let $r_m$ be a sequence of all rational positive numbers, then
$$
Mf=\sup_m|f|*h_{r_m}.
$$
For $g_k=\max_{1\leq m\leq k}|f|*h_{r_m}$ we have
\begin{align}\label{nierb2}
|D_ig_k(x)|\leq \max_{1\leq m\leq k}\left|D_i(|f|*h_{r_m})(x)\right|\leq M(D_i|f|)(x)=M(D_if)(x)
\end{align}
for almost all $x\in\rn$ and all $k\in\n$. Taking into account (\ref{nierb2}) and Theorem \ref{lfmax}, we have 
\begin{align*}
\|g_k\|_{1,\phi}\leq\|Mf\|_{\phi}+\sum_{i=1}^n\|M(D_if)\|_{\phi}\leq C\|f\|_{1,\phi}.
\end{align*}
In this way we have proven the boudedness of the sequence $g_k$ in $\wfr$. Since $\wfr$ is reflexive, we have a subsequence of $g_k$ (still denoted as $g_k$) which is weakly convergent in $\wfr$. On the other hand, we have that $g_k$ is the increasing sequence converging pointwisely to $Mf$. Thus, we obtain $g_k\rightharpoonup Mf$ in $\wfr$. Therefore, from inequality (\ref{nierb2}) we get  
$$
\int_A D_iMf(x)\,dx\leq \int_AM(D_if)(x)\,dx \quad \text{ and } \quad \int_A D_iMf(x)\,dx \geq -  \int_AM(D_if)(x)\,dx,
$$
for every measurable set $A$ with finite measure. This finishes the proof.
\end{proof}

\section{Auxiliary results}
For $f\in L_{\loc}^1(\rn)$ and $x\in\rn$ we define (see \cite{Luiro})
$$
\R f(x)=\left\{ r\geq 0\,:\, \exists(\{r_k\}\subset (0,\infty))\,r_k\to r\wedge Mf(x)=\lim_{k\to\infty}\vint_{B(x,r_k)}|f(y)|\,dy\right\}.
$$

\begin{prop}\label{otprop}
Let $\phi\in \fiwr$ satisfies \adec\ and \Az, then for every $f\in\lfr$, the following statements hold.
\begin{itemize}
\item[(i)] For all $x\in\rn$ the set $\R f(x)$ is nonempty.
\item[(ii)] For all $x\in\rn$ and $r>0$ such that $r\in\R f(x)$  the equality $$Mf(x)=\vint_{B(x,r)}|f(y)|\, dy$$ holds.
\item[(iii)] For almost all $x\in\rn$ if $0\in\R f(x)$, then $$Mf(x)=|f(x)|.$$
\end{itemize}
\end{prop}
\begin{proof}

 (i) The statement is obvious for $f=0$. Thus, we assume that $f\not=0$. In this case $Mf(x)>0$ for any $x \in \mathbb{R}^n$.
Next, for $x\in\rn$ there exists a sequence $\{r_n\}\subset (0,\infty)$ such that 
\begin{equation}\label{otprop1}
\lim_{n\to\infty}\vint_{B(x,r_n)}|f(y)|\, dy=Mf(x).
\end{equation}
Let $p\in (1,\infty)$ be such that $\phi$ satisfies $\adec_p$. We know that $\phi^*$ satisfies $\ainc_{p'}$ \cite[Proposition 2.4.9]{genOrlicz}, and let $a\in [1,\infty)$ be a constant from the definition of $\ainc_{p'}$. We also know that $\phi^*$ satisfies \Az \cite[Lemma 3.7.6]{genOrlicz}. Let $\beta\in (0,1]$ be a constant for $\phi^*$ from Proposition \ref{azrow}.

Let us take $r>0$ such that $|B(x,r)| \geq 1$, then by the H\"older inequality, inequality (\ref{osznorm}) and Proposition \ref{azrow} we get the string of inequalities
\begin{align}\label{otin1}
\vint_{B(x,r)}|f(y)|\,dy&\leq \frac{2}{\beta|B(x,r)|}\|f\|_\phi\|\beta\chi_{B(x,r)}\|_{\phi^*}\nonumber\\
&\leq \frac{2}{\beta|B(x,r)|}\|f\|_\phi\max\{(a\rho_{\phi^*}(\beta\chi_{B(x,r)}))^{\frac{1}{p'}},1\} \nonumber\\
&\leq\frac{2 a^{\frac 1{p'}}}\beta\|f\|_{\phi}|B(x,r)|^{\frac{1}{p'}-1}.
\end{align}
The above inequality yields that $\vint_{B(x,r)}|f(y)|\,dy\to 0$ when $r\to\infty$.
 Hence, by (\ref{otprop1}) and since $Mf(x)>0$, we get that $\{r_n\}$ is bounded. Thus, there exist a subsequence $\{r_{n_k}\}$ and $r\in [0,\infty)$ such that $r_{n_k}\to r$. Therefore, we obtain $r\in\R f(x)$. 

(ii) follows from the continouity of integral with respect to measure, and (iii) is a direct consequence of the Lebesgue differentiation Theorem, since $f\in L^1_{loc}(\mathbb{R}^n)$.
\end{proof}

\begin{lem}\label{ogrrf}
Let $\phi\in\fiwr$ satisfies \adec\,  \Az\ and let $R>0$. Then, for $f\in\lfr$ such that $f\not =0$ we have
$$
\sup\left\{r\,:\,\exists (x\in B(0,R))\, r\in\R f(x)\right\}<\infty.
$$
\end{lem}
\begin{proof}
Let us suppose $\sup \left\{r\,:\,\exists (x\in B(0,R))\, r\in\R f(x)\right\}=\infty$, then there exist sequences $\{r_k\}\subset (0,\infty)$ and $\{x_k\}\subset B(0,R)$ such that 
$$
r_k\to\infty\quad\textrm{and}\quad r_k\in\R f(x_k).
$$
Let $\widetilde{R}>0$ be such that $\int_{B(0,\widetilde{R})}|f|dx>0$ and define $\widehat{R}=\max\left\{R,\widetilde{R}\right\}$. Let us observe that for any $x\in B(0, R)$ we have
$$
Mf(x)\geq\vint_{B(x,2\widehat{R})}|f(y)|\,dy\geq\frac{1}{|B(0,2\widehat{R})|}\int_{B(0,\widehat{R})}|f(y)|\,dy>0.
$$
Therefore, taking into account Proposition \ref{otprop} and inequality (\ref{otin1}), we get
\begin{align*}
\frac{1}{|B(0,2\widehat{R})|}\int_{B(0,\widehat{R})}|f(y)|\,dy &\leq Mf(x_k)=\vint_{B(x_k,r_k)}|f(y)|\,dy \\
&\leq \frac{2 a^{\frac 1{p'}}}\beta\|f\|_{\phi}|B(x_k,r_k)|^{\frac{1}{p'}-1}\to 0,
\end{align*}
and this gives us a contradiction.
\end{proof}
\begin{cor}
Let $\phi\in\fiwr$ satisfies \adec\,  \Az\ and let $f\in\lfr$ such that $f\not =0$. Then, for every Lebesgue point $x$ of $f$ the set $\R f(x)$ is compact.
\end{cor}
\begin{lem}\label{lemogrgr}
Let $\phi\in\fiwr$ satisfies \adec\, and \Az. Let $f, f_m\in\lfr$ be such that $f_m\to f$ in $\lfr$. If $f\not=0$ and $f_m\not = 0$, then for all $R>0$ there exists $\fR>0$ such that 
$$
\sup\left\{r\,:\,\exists (x\in B(0,R))\, r\in\R f_m(x)\right\}\leq \fR
$$
holds for all $m\in\n$.
\end{lem}
\begin{proof}
Let us suppose that the thesis of the lemma dos not hold. 
Thus, there exist a sequence $\{m_k\}\subset\n$, a sequence of points $\{x_k\}\subset B(0,R)$ and a sequence $\{r_k\}$ such that for all $k$ we have
$$
r_k\in\R f_{m_k}(x_k)\quad\textrm{ and }\quad r_k\to\infty.
$$ 
Let us observe that if $\{m_k\}$ is bounded, then there exists a subsequence $\{m_{k_l}\}$ such that $f_{m_{k_l}}$ is constant and  $r_{k_l}\in\R f_{m_{k_l}}(x_k)$ with $r_{k_l}\to\infty$. But this contradicts to Lemma \ref{ogrrf}.

Therefore, we assume that $\{m_k\}$ is unbounded. Hence, there exists an increasing subsequence $\{m_{k_l}\}\rightarrow \infty$. 
Let us denote $M=\sup_m\|f_m\|_{\phi}$ and let $p \in (1,\infty)$ be such $\phi$ satisfies $\adec_p$. By inequality (\ref{otin1}) 
we have
\begin{align}\label{ogrfmgr}
M f_{m_{k_l}}(x_{k_l})=\vint_{B(x_{k_l},r_{k_l})}|f_{m_{k_l}}|(y)dy\leq C\|f_{m_{k_l}}\|_{\phi}|B(0,r_{k_l})|^{-1/p}\leq CM|B(0,r_{k_l})|^{-1/p}\to 0.
\end{align}
Let $\wR$ be such that $\int_{B(0,\wR)}|f|dx>0$ and define $\widehat{R}=\max\left\{R,\wR\right\}$. For all $m\in\n$ and for all $x\in B(0,R)$ we have
\begin{align}\label{ogrfmnier}
Mf_m(x)&\geq\frac{1}{B(0,2\widehat{R})}\int_{B(0,\widehat{R})}|f_m(y)|dy\nonumber\\
&\geq \frac{1}{B(0,2\widehat{R})}\int_{B(0,\widehat{R})}|f(y)|dy-\frac{1}{B(0,2\widehat{R})}\int_{B(0,\widehat{R})}|f(y)-f_m(y)|\,dy.
\end{align}
Since $f_m \rightarrow f$ in $\lfr$, by Proposition \ref{lfzan} there exists $N\in\n$ such that for all $m\geq N$ we have
$$
\frac{1}{B(0,2\widehat{R})}\int_{B(0,\widehat{R})}|f(y)-f_m(y)|\,dy\leq \frac{1}{2B(0,2\widehat{R})}\int_{B(0,\widehat{R})}|f(y)|\,dy.
$$
Taking into account the above inequality and (\ref{ogrfmnier}), we obtain for $m\geq N$ the following bound 
$$
Mf_m(x)\geq\frac{1}{2B(0,2\widehat{R})}\int_{B(0,\widehat{R})}|f(y)|\,dy.
$$
Thus, since $m_{k_l}\geq N$ for large $l$, the above inequality contradicts to (\ref{ogrfmgr}).
\end{proof}

\begin{lem}\label{lemut}
Let $\phi\in\fiwr$ satisfies \Az\ and \adec, $f\in\lfr$ such that $f \neq 0$ and $R>0$. Moreover, let $\fR>{R_0}+R$, where  
$$
R_0=\sup\left\{r\,:\,\exists (x\in B(0,R))\, r\in\R f(x)\right\},
$$ 
and let $g\in L^1_\loc(\rn)$ be such  
$$
g(x)=f(x)\textrm{ for almost all }x\in B(0,\fR)\quad\textrm{and}\quad |g(x)|\leq|f(x)|\textrm{ for almost all }x\in\rn\setminus B(0,\fR),
$$
then the following 
$$
\R f(x)=\R g(x)\qquad\textrm{and}\qquad Mf(x)=Mg(x),
$$
hold for all $x\in B(0,R)$.
\end{lem}
\begin{proof}
Let $r\in \R f(x)$, then there exists $\{r_m\}\subset (0,\infty)$ such that 
$$
r_m \rightarrow r \qquad\textrm{and}\qquad\lim_{m \rightarrow \infty} \vint_{B(x,r_m)} |f(y)|\, dy=Mf(x).
$$
Since $r\leq R_0$, we can assume that $r_m\leq \fR-R$. Therefore, we obtain $B(x,r_m)\subset B(0,\fR)$, and we get 
$$
Mg(x)\geq\lim_{m\rightarrow \infty}  \vint_{B(x,r_m)} |g(y)| \, dy = \lim_{m \rightarrow \infty}  \vint_{B(x,r_m)} |f (y)| \, dy =  Mf(x) \geq Mg(x).
$$
This yields 
$$
\lim_{m \rightarrow \infty}  \vint_{B(x,r_m)} |g (y)| \, dy =  Mg(x).
$$
Thus, $r\in \R g(x)$ and we finished the proof of the inclusion $ \R f(x) \subset \R g(x)$.

Now, let  $\rho\in \R g(x)$ and let us take $\{\rho_m\}\subset (0,\infty)$ such that 
	$$
	\rho_m \rightarrow \rho \qquad\textrm{and}\qquad\lim_{m \rightarrow \infty} \vint_{B(x,\rho_m)}|g(y)|\, dy=Mg(x).	
	$$
	We also take\footnote{We can do it since by Proposition \ref{otprop} the set $\R f(x)$ is nonempty.}  $r\in \R f(x)$ and a sequence $\{r_n\}$ such that 
	$$
	r_m \rightarrow r \qquad\textrm{and}\qquad\lim_{m \rightarrow \infty} \vint_{B(x,r_m)}|f(y)|\, dy=Mf(x).
	$$
	Since $r\leq R_0$, we can also assume that $r_m\leq \fR-R$ for all $m$.
	We obtain 
	\begin{align*}
	Mf(x)&\geq \limsup_{m\to\infty}\vint_{B(x,\rho_m)}|f(y)|\,dy\geq \liminf_{m\to\infty}\vint_{B(x,\rho_m)}|f(y)|\,dy\geq\lim_{m\to\infty}\vint_{B(x,\rho_m)}|g(y)|\,dy=Mg(x)\\
& \geq \limsup_{m\to\infty}\vint_{B(x,r_m)}|g(y)|\,dy=\lim_{m\to\infty}\vint_{B(x,r_m)}|f(y)|\,dy=M f(x).
	\end{align*}
	It yields $\rho \in\R f(x)$, and we get $\R f(x)=\R g(x)$. The equality $Mf(x)=Mg(x)$ follows from the proven inequalities.
\end{proof}

\begin{lem} \label{lem:otoczka}
	Let $\phi\in\fiwr$ satisfies \Az, \adec and \ainc. If $f_m, f \in \lfr$  and $f_m \rightarrow f$ in $\lfr$, then for all $R>0$ and $\lambda>0$ the set $\left\{ x \in B(0,R) \, : \, \R f_m(x) \not\subset \R f(x)_{(\lambda)}  \right\}$ is measurable and\footnote{For nonempty set $A\subset \mathbb{R}^n$ and $\lambda \geq 0$ we denote $A_{(\lambda)}=\{x\in \mathbb{R}^n: dist(x,A) \leq \lambda\}$.}
	\[
	 \lim_{m \rightarrow \infty}\left| \left\{ x \in B(0,R) \, : \, \R f_m(x) \not\subset \R f(x)_{(\lambda)}  \right\} \right| = 0 .
	\]
\end{lem}
\begin{proof} 
If $f=0$, then the lemma is obviously true. It is easy to see that if $f\neq 0$, then $f_m \neq 0$ for large $m$. Therefore, we can assume that $f\neq 0$ and $f_m \neq 0$.
From Lemma \ref{ogrrf} and Lemma \ref{lemogrgr} there exists $\fR>0$ such that for all $x\in B(0,R)$, $m$ and $r\in\R f_m(x) \cup \R f(x)$ we have $r\leq\fR$.
We define
$$
g=\chi_{B(0,R+\fR)}f\quad\textrm{and}\quad g_m=\chi_{B(0,R+\fR)}f_m.
$$
It is obvious that $g_m\to g$ in $\lfr$ and in $L^1(\rn)$ by Proposition \ref{lfzan}. By Lemma \ref{lemut} we have 
$$
\R g(x)=\R f(x)\quad\textrm{and}\quad \R g_m(x)=\R f_m(x).
$$
for any $x\in B(0,R)$. 
Thus, it yields
\begin{align}\label{meas}
\left\{ x \in B(0,R) \, : \, \R f_m(x) \not\subset \R f(x)_{(\lambda)}  \right\}=\left\{ x \in B(0,R) \, : \, \R g_m(x) \not\subset \R g(x)_{(\lambda)}  \right\}.
\end{align}
From Lemma \ref{lem:otoczka} for $L^1$ -space, which was proven in \cite{Luiro}, we obtain 
$$
\left|\left\{ x \in B(0,R) \, : \, \R g_m(x) \not\subset \R g(x)_{(\lambda)}  \right\}\right|\to 0,
$$
and this finishes the proof. Let us stress that the measurability of the set $\left\{ x \in B(0,R) \, : \, \R f_m(x) \not\subset \R f(x)_{(\lambda)}  \right\}$ follows from (\ref{meas}) and the  measurability of the set $\left\{ x \in B(0,R) \, : \, \R g_m(x) \not\subset \R g(x)_{(\lambda)}  \right\}$, which was proven in \cite{Luiro}.
\end{proof}
\section{Continuity of the maximal operator}

The following theorem is the main result of the paper.
\begin{tw}
Let us assume that $\phi\in\fiwr$ satisfies \Az, \Aj, \Ad, \ainc\ and \adec, then the maximal operator $$M\colon\wfr\to\wfr$$ is continuous.
\end{tw}
\begin{proof}
First of all we shall prove the formula for weak derivatives of the maximal function.
\begin{lem}\label{wz}
Let us assume that $\phi\in\fiwr$ satisfies \Az, \Aj, \Ad, \ainc\ and \adec. 	If $f \in \wfr$,  then for almost all $x \in \mathbb{R}^n$ and for all $i=1,\ldots,n$ we have
	\begin{align}\label{wzteza}
	\begin{split}
		 D_iM f(x) &= \vint_{B(x,r)} D_i|f|(y)\, dy\textrm{ for all $r \in \R f(x)$, $r>0$ and}\\ 
		  D_i Mf(x) &= D_i|f|(x) \textrm{ if $0 \in \R f(x)$.}
		  \end{split}
	\end{align}
\end{lem}
\begin{proof}
We can assume that $f \neq 0$. Let $R>0$ and let $\fR=R+R_0+1$, where
$$
R_0=\sup\left\{r\,:\,\exists (x\in B(0,R))\, r\in\R f(x)\right\}.
$$
 Let us take $\psi\in C^{\infty}_{\textnormal{c}}(\rn)$ such that $0\leq\psi\leq 1$ and $\psi(x)=1$ for all $x\in B(0,\fR)$.

Next, we define $g=f\psi$. Let $p \in (1, \infty)$ be such that $\phi$ satisfies $\ainc_p$. It is obvious that $g\in\wfr$ and $g\in W^{1,p}(\rn)$ by Proposition \ref{lfzan}. By Lemma \ref{lemut}, for $x\in B(0,R)$ we have
\begin{align}\label{wzrow}
\R f(x)=\R g(x)\quad\textrm{and}\quad Mf(x)=Mg(x).
\end{align}
Since Lemma \ref{wz} holds for $W^{1,p}$ spaces (see \cite{Luiro}), we have that (\ref{wzteza}) holds for $g$. Therefore, having in mind (\ref{wzrow}) and the form of $g$ we easily get the thesis.
\end{proof}

Let $f_m\to f$ in $\wfr$. We shall show that $Mf_m\to Mf$ in $\wfr$. Due to Theorem \ref{lfmax}, it is enough to show that $D_if_m\to D_if$ in $\lfr$ for all $1\leq i\leq n$.

Let us fix $\epsilon>0$, then by Proposition \ref{propjedn} there exist $R>0$ and $\lambda >0$ such that $\|MD_if\|_{\pc,\rn\setminus B(0,R)}<\epsilon/2$ and 
$\|MD_if\|_{\pc,A}<\epsilon/2$ whenever $|A|<\lambda$. For every $x\in\rn$ we define functions $u_x, u^m_x \colon [0,\infty)\to \r$ as follows
\begin{align*}
u_x(r)=\begin{cases}
\vint_{B(x,r)}D_if&\textrm{ if }r>0,\\
D_if(x)&\textrm{ if }r=0
\end{cases}
\quad
\quad
u^m_x(r)=\begin{cases}
\vint_{B(x,r)}D_i f_m&\textrm{ if }r>0,\\
D_if_m(x)&\textrm{ if }r=0.
\end{cases}
\end{align*}
By the Lebesgue differentiation Theorem and since $\lim_{r\rightarrow \infty} u_x (r)=0$, we have that $u_x$ is uniformly continuous for almost all $x\in\rn$ and $u_x$ is continuous on $(0,\infty)$ for all $x\in\rn$. In other words the set 
\[
\mathcal{N} =  \left\{ x \in B(0,R)\, : \,  u_x \text{ is not uniformly continuous}\right\}
\]
is a null set. 
Next, for $\delta >0$ we define
\[
C_\delta = \left\{ x \in B(0,R) \,:\, \exists ( r_1,r_2\in [0,\infty))\ |r_1-r_2|\leq \delta \textrm{ and }|u_x(r_1) - u_x(r_2)| > \epsilon/ \|\chi_{B(0,R)}\|_{\phi}\right\}.
\]
Since $\bigcap\limits_{\delta>0} C_\delta \subset \mathcal{N}$ and $C_{\delta_1} \subset C_{\delta_2}$ when $\delta_1 < \delta_2$, we can choose\footnote{Since 
$ C_\delta=\bigcup_{\substack{ r_1,r_2\in [0,\infty)\cap\mathbb{Q} \\ |r_1-r_2|\leq\delta}}\left\{ x \in B(0,R) \,:\, |u_x(r_1) - u_x(r_2)| >\epsilon/\|\chi_{B(0,R)}\|_{\phi}\right\}
$, we have that $C_{\delta}$ is measurable.
} $\delta_0$ such that $|C_{\delta_0}| < \frac{\lambda}{2}$. Due to Lemma \ref{lem:otoczka} there exists $m_0$ such that
\[
\left| \left\{ x \in B(0,R) \, : \, \R f_m(x) \not\subset \R f(x)_{(\delta_0)}  \right\} \right| < \frac{\lambda}{2}, \quad \text{ for } m \geq m_0.
\]
We denote the set in the above formula by $C^m$. Let us also denote by $\Psi$ a set of $x\in\rn$ such that (\ref{wzteza}) is satisfied for $f$ and $f_m$ for all $m$ and $x$ is the Lebesgue point of $f, f_m, D_i f$ and $D_i f_m$ for all $m$. From Lemma \ref{wz} and the Lebesgue differentiation Theorem we have $|\rn\setminus\Psi|=0$.

Now, let us take $m \geq m_0$, $x\in \Psi$, then for $r_1 \in \R f_m(x)$ and $r_2 \in \R f(x)$ we have
\begin{align}\label{conin1}
\left| D_iMf_m(x) - D_iMf(x)  \right| & = \left|u_x^m(r_1) - u_x(r_2)  \right| \nonumber \\
& \leq  \left| u_x^m(r_1) - u_x(r_1)   \right|   + \left| u_x(r_1) - u_x(r_2)   \right| \nonumber\\
& \leq M\left( D_if_m - D_if \right)(x) +  \left|u_x(r_1) - u_x(r_2) \right|.  
\end{align}
Next, if $x \in (\mathbb{R}^n \setminus (C_{\delta_0} \cup C^m\cup( \mathbb{R}^n \setminus B(0,R)))\cap \Psi$ we can find $r_1 \in \R f_m(x)$ and $r_2 \in \R f(x)$ such, that\footnote{The compactness of $\R f(x)$ is used in this place.} $|r_1-r_2| \leq \delta_0$. Therefore, by the definition of $C_{\delta_0}$, we obtain
\[
\left| u_x(r_1) - u_x(r_2) \right| \leq \epsilon/\|\chi_{B(0,R)}\|_{\phi} .
\]
If $x \in  (C_{\delta_0} \cup C^m\cup( \mathbb{R}^n \setminus B(0,R))\cap \Psi$, then we use the following bound
\[
\left|  u_x(r_1) - u_x(r_2) \right| \leq 2MD_if(x).
\]
Combing those inequalities with (\ref{conin1}) we get
\begin{align*}
\left\| D_iMf_m - D_iMf  \right\|_{\phi,\mathbb{R}^n} &\leq \left\|M(D_if_m - D_if)  \right\|_{\phi,\mathbb{R}^n} + \| \tfrac{\epsilon}{\|\chi_{B(0,R)}\|_\phi } \|_{\phi,B(0,R)} \\
& +  \left\|2 MDf \right\|_{\phi,C_{\delta_0}\cup C^m} + \left\|2 MDf \right\|_{\phi,\rn\setminus B(0,R)} \\
& \leq   C\left\|D_if_m - D_if\right\|_{\phi,\mathbb{R}^n} + 3 \epsilon.
\end{align*}
Finally, since $D_if_m\to D_if$ in $\wfr$, the proof follows. 
\end{proof}
\bibliographystyle{plain}
\bibliography{conmaxopbib}
\end{document}